\documentclass[11pt]{article}
\usepackage{amsmath}
\usepackage{amssymb}
\usepackage{amsfonts}
\usepackage{amsthm}
\usepackage{mathrsfs}
\usepackage{latexsym}
\usepackage{epsfig}

\newtheorem{thm}{Theorem}
\newtheorem{lma}{Lemma}

\newtheorem{obs}[lma]{Observation}
\newtheorem{crly}[thm]{Corollary}

\newtheorem*{ntn}{Notation}

\begin{document}
\title{On the Mixing Set with a Knapsack Constraint}
\author{Ahmad Abdi and Ricardo Fukasawa\\
Department of Combinatorics and Optimization†\\
University of Waterloo\\
\texttt{\{a3abdi,rfukasaw\}@math.uwaterloo.ca}}
\date{\today}
\maketitle

\begin{abstract} 
The mixing set with a knapsack constraint arises as a substructure in mixed-integer programming 
reformulations of chance-constrained programs with stochastic right-hand-sides over a finite discrete distribution.
Recently, Luedtke et al. (2010) and K\"u\c{c}\"ukyavuz (2012) studied valid inequalities for such sets. 
However, most of their results were focused on the equal probabilities case (equivalently when the knapsack reduces to a cardinality constraint), with only minor results in the general case. In this paper, we focus on the general probabilities case (general knapsack constraint). We characterize the valid inequalities that do not come 
from the knapsack polytope and use this characterization to generalize the inequalities previously derived
for the equal probabilities case. We also show that one can separate over a large class of inequalities in polynomial time.
\end{abstract}

\section{Introduction}\label{intro}

Many optimization problems in real world applications allow to some extent a number of violated constraints, which results in a decrease in the quality of service (QoS) and also a decrease in the cost of production. These optimization problems have been a main motive to study probabilistic (in particular, chance-constrained) programming. A difficulty when dealing with these optimization problems is that the feasible region is not necessarily convex. In this paper, we consider mixed-integer programming (MIP) reformulations of chance-constrained programs with joint probabilistic constraints in which the right-hand-side vector is random with a finite discrete distribution. This model was first proposed in Sen \cite{sen1992}, studied in Ruszczy\'{n}ski \cite{ruszczynski2002}, and extended by Luedtke et. al \cite{luedtke2010} and K\"{u}\c{c}\"{u}kyavuz \cite{kucukyavuz2012}. This reformulation gives rise to a mixing-type set \cite{gunluk2001} subject to an additional knapsack constraint, which is the focus of this paper.

Formally, consider the following chanced-constrained programming problem \[\begin{array}{ccl}$(PLP)$ &\min & c^T x \\
&{\rm s.t.} & {\bf P}(f(x) \geq \xi) \geq 1-\epsilon \\ & & x \in X, \end{array} \] where $X\subset \mathbb{R}^n$ is a polyhedron, $f:X\to \mathbb{R}^d_+$ is a linear function, $\xi$ is a random variable in $\mathbb{R}^d$ with finite discrete distribution, $\epsilon \in (0,1)$, and $c\in \mathbb{R}^n$. Suppose that $\xi$ takes values from $\xi_1,\ldots,\xi_n$ with probabilities $\pi_1,\ldots,\pi_n$, respectively. We may assume that $\xi_j\geq 0$ for all $j\in [n]:=\{1,\ldots,n\}$. (Otherwise, replace $\xi_j$ by $\xi_j-\xi'$ and reset $f(x):=f(x)-\xi'$, where $\xi'$ is chosen so that $\xi_j\geq \xi'$ for all $j\in [n]$.) Also, by definition, $\pi_j>0$ for each $j\in [n]$ and $\sum_{j=1}^{n}{\pi_j}=1$. We can reformulate the chance constraint in (PLP) using linear inequalities and auxiliary binary variables as follows: let $z\in \{0,1\}^n$ where $z_j=0$ guarantees that $f(x)\geq \xi_j$. Then (PLP) is equivalent to \[ \begin{array}{ccl}$(PLP)$ &\min & c^T x \\
&{\rm s.t.} & y=f(x) \\ & & y+z_j\xi_j\geq \xi_j ~~~~ \forall j\in [n]\\ & & \sum_{j=1}^n\pi_jz_j\leq \epsilon \\ & & z\in \{0,1\}^n \\ & & x \in X. \end{array} \] Observe that we may assume that $\pi_j\leq \epsilon$ for all $j\in [n]$, for if $\pi_j>\epsilon$ for some $j\in [n]$, then we must have $z_j=0$ for all feasible solutions $(x,y,z)$ to the above system, and so may as well drop the index $j$. Now let $$\mathcal{D}:=\left\{(y,z)\in \mathbb{R}^d_+ \times \{0,1\}^n:\sum_{j=1}^n{\pi_jz_j}\leq \epsilon, ~y+\xi_jz_j\geq \xi_j ~ \forall j\in [n]\right\}.$$ Then (PLP) can be rewritten as  \[ \begin{array}{ccl}$(PLP)$ &\min & c^T x \\
&{\rm s.t.} & f(x)\in \text{proj}_y \mathcal{D}\\ & & x \in X. \end{array}\] This motivates us to study the set $\mathcal{D}$. For each $k\in [d]$, let $$\mathcal{D}_k:=\left\{(y_k,z)\in \mathbb{R}_+ \times \{0,1\}^n:\sum_{j=1}^n{\pi_jz_j}\leq \epsilon, ~y_k+\xi_{jk}z_j\geq \xi_{jk} ~ \forall j\in [n]\right\}.$$ Then observe that $$\mathcal{D}=\bigcap_{k\in [d]}\left\{(y,z)\in  \mathbb{R}^d_+ \times \{0,1\}^n: (y_k,z)\in \mathcal{D}_k\right\}.$$ Therefore, in order to study the set $\mathcal{D}$, a first step is to study the lower dimensional sets $\mathcal{D}_k$.

Fix $k\in [d]$ and for notational convenience, let $h_j:=\xi_{jk}$ for each $j\in [n]$. Let $\sum_{j\in [n]}{a_j z_j}\leq p$ be a valid inequality for $\mathcal{D}_k$ where $a\in \mathbb{R}_+^{n}$, $p\in \mathbb{R}_+$, $a_j\leq p$ for all $j\in [n]$, and $\sum_{j\in [n]}a_j>p$. Observe that this inequality may be the knapsack constraint $\sum_{j=1}^n{\pi_jz_j}\leq \epsilon$. Now let $$\mathcal{Q}:=\left\{(y,z)\in \mathbb{R}_+ \times \{0,1\}^n:\sum_{j\in [n]}{a_jz_j}\leq p, ~y+h_jz_j\geq h_j ~ \forall j\in [n]\right\}.$$ Note that the assumption that $a_j\leq p$ for all $j\in [n]$ implies that $\mathcal{Q}$ is a full-dimensional set. (The points $(h_1+1,0), (h_1,e_1),\ldots,(h_1,e_n)$ are in $\mathcal{Q}$, where $e_j$ is the $j$-th $n$-dimensional unit vector.) Also, note that the assumption that $\sum_{j\in [n]}a_j>p$ implies that $y\geq h_n\geq 0$ for all $y\in \mathcal{Q}$. Observe that the set $\mathcal{Q}$ contains as a substructure the intersection of a mixing set, first introduced by G\"{u}nl\"{u}k and Pochet \cite{gunluk2001}, and a knapsack constraint. Various structural properties of conv$(\mathcal{Q})$ were studied in \cite{luedtke2010} and \cite{kucukyavuz2012} when the knapsack constraint $\sum_{j\in [n]}a_jz_j\leq p$ is just a cardinality constraint. In \cite{luedtke2010}, a characterization of all valid inequalities of conv$(\mathcal{Q})$ was given, and in both \cite{luedtke2010} and \cite{kucukyavuz2012}, explicit classes of facet-defining inequalities were introduced. \\

\noindent Outline of Our Work

\noindent In this paper, we do not make any assumptions on the knapsack constraint. In Sect$.$ \ref{char}, we characterize the set of all valid inequalities for conv$(\mathcal{Q})$, and give a general cutting plane generating algorithm. In Sect$.$ \ref{fdisection}, we give necessary conditions for facet-defining inequalities of conv$(\mathcal{Q})$. In Sect$.$ \ref{explicitfdisub}, we introduce an explicit class of facet-defining inequalities that subsumes the facet-defining inequalities found in \cite{luedtke2010} and \cite{kucukyavuz2012}. Finally, in Sect$.$ \ref{heuristicsep}, using our ideas regarding characterization of all valid inequalities, we introduce a polynomial time heuristic separation algorithm for conv$(\mathcal{Q})$.

\section{The Coefficient Polyhedron $\mathcal{G}$}\label{char}

Let $$\mathcal{P}:=\left\{z\in \{0,1\}^{{n}}:\sum_{j\in [n]}a_jz_j\leq p \right\}.$$ Observe that $\mathcal{P}=\text{proj}_z\mathcal{Q}$ and $\text{conv}(\mathcal{P})=\text{proj}_z (\text{conv}(\mathcal{Q}))$. The focus of this paper is to study the class of valid inequalities for $\text{conv}(\mathcal{Q})$ that do not arise from $\text{conv}(\mathcal{P})$. We will first show that any such inequality has a particular form.

\begin{lma}\label{genform} Suppose that \begin{eqnarray} \gamma y+\sum_{j\in [n]}\alpha_jz_j\geq \beta\label{gvdi} \end{eqnarray} is a valid inequality for $\emph{conv}(\mathcal{Q})$ for some $\alpha\in \mathbb{R}^n, \gamma,\beta\in \mathbb{R}$. Then $\gamma\geq 0$. Moreover, if $\gamma=0$ then $(\ref{gvdi})$ is a valid inequality for $\emph{conv}(\mathcal{P})$. \end{lma}

\begin{proof} Observe that $(1,0)\in \text{cone}(\text{conv}(\mathcal{Q}))$. This implies that $\gamma\geq 0$. Moreover, since $\text{conv}(\mathcal{P})=\text{proj}_z (\text{conv}(\mathcal{Q}))$, it follows that if $\gamma=0$ then $(\ref{gvdi})$ is a valid inequality for $\text{conv}(\mathcal{P})$.\end{proof}

As a side note, the problem of finding a characterization for the class of all valid inequalities of conv$(\mathcal{P})$ is a very difficult problem and it has been extensively studied; the seminal works may be found in \cite{balas1975, balas1978, hammer1975, wosley1975}. Recall that only those valid inequalities for $\text{conv}(\mathcal{Q})$ are of interest that do not come from $\text{conv}(\mathcal{P})$. As a result, by rescaling the coefficients, if necessary, we may assume that (\ref{gvdi}) has the following form: \begin{eqnarray} y+\sum_{j\in [n]}\alpha_jz_j\geq \beta.\label{gvdi2} \end{eqnarray}

It turns out that it is possible to characterize the inequalities of type (\ref{gvdi2}) by considering the set of all vectors $(\alpha,\beta)$ that give valid inequalities of type (\ref{gvdi2}). In this section, we will explicitly find this set, which happens to be a polyhedron. This polyhedron and its formulation will help us throughout the paper with various results on the structure of conv$(\mathcal{Q})$.

Let $\nu:=\max\{k:\sum_{j\leq k}a_j\leq p\}$. Notice that if $(y^*,z^*)\in \mathcal{Q}$ then $y^*\geq h_{\nu+1}$. Define, for each $0\leq k\leq \nu$, the knapsack set $$\mathcal{P}_k:=\left\{z\in \{0,1\}^{[n]}:\sum_{j> k}a_jz_j\leq p-\sum_{j\leq k}a_j \right\}.$$ Observe that $\mathcal{P}=\mathcal{P}_0 \supset \mathcal{P}_1\supset \cdots \supset \mathcal{P}_{\nu}$. Define the polyhedron $$\mathcal{G}:=\left\{(\alpha,\beta)\in \mathbb{R}^n\times \mathbb{R}: (\ref{Gconstp})\right\}$$ where \begin{eqnarray}
\sum_{j\leq k}\alpha_j+\sum_{j> k}\alpha_jz^*_j+h_{k+1} \geq \beta ~~~~\forall~ z^*\in \mathcal{P}_k, \forall ~0\leq k\leq \nu.\label{Gconstp} \end{eqnarray} The following theorem proves that $\mathcal{G}$ is the desired set, and is one of the main results of this section.

\begin{thm}\label{validthm} Choose $(\alpha,\beta)\in \mathbb{R}^n\times \mathbb{R}$. Then $(\ref{gvdi2})$ is a valid inequality for \emph{conv}$(\mathcal{Q})$ if and only if $(\alpha,\beta)\in \mathcal{G}$.\end{thm}

We refer to $\mathcal{G}$ as the {\it coefficient polyhedron} of $\mathcal{Q}$. Before proving the above lemma, we would like to point out  that $\mathcal{G}$ has an alternate formulation with $O(n)$ non-linear inequalities. For $\alpha\in \mathbb{R}^n$ and $0\leq k\leq \nu$, define the {\it minimizer} function \begin{eqnarray} f_k(\alpha):=\min \left\{\sum_{j> k}\alpha_jz_j:z\in \mathcal{P}_k \right\}. \label{f}\end{eqnarray} Then it is easy to see that $$\mathcal{G}=\left\{(\alpha,\beta)\in \mathbb{R}^n\times \mathbb{R}:\sum_{j\leq k}\alpha_j+ f_k(\alpha) +h_{k+1}\geq \beta ~\forall ~0\leq k\leq \nu\right\}.$$ As one may expect, the minimizers $f_k$ will play a central role  when looking for structural properties of $\mathcal{G}$. For instance, as long as the minimizers $f_k$ can be efficiently solved to optimality, the polyhedron $\mathcal{G}$ can be efficiently described by a system of $O(n)$ non-linear inequalities. Notice that this is not too surprising since if $f_0$ can be efficiently solved for all $\alpha\in \mathbb{R}$, then the reverse polar of conv$(\mathcal{P})$ can be efficiently described, and this yields an efficient method to obtain valid inequalities for conv$(\mathcal{Q})$, which correspond to points in $\mathcal{G}$. For more details on this approach, see Sen \cite{sen1992}. However, note that our purpose here is {\it not} necessarily finding fast and practical algorithms to obtain valid inequalities but rather obtaining structural results about conv$(\mathcal{Q})$. The coefficient polyhedron $\mathcal{G}$ and the minimizers $f_k$ will be useful to better understand the structure of the polyhedron conv$(\mathcal{Q})$.

It is now time to prove Theorem \ref{validthm}.

\begin{proof} Suppose that \begin{eqnarray} y+\sum_{j\in [n]}\alpha_jz_j\geq \beta \label{genvdi}\end{eqnarray} is a valid inequality of conv$(\mathcal{Q})$ for some $(\alpha,\beta)\in \mathbb{R}^n\times \mathbb{R}$. Take $0\leq k\leq \nu$. Let $z^*\in \{0,1\}^n$ be an optimal solution to (\ref{f}) with $z^*_j=1$ for all $1\leq j\leq k$. Note that such an optimal solution exists due to how $\mathcal{P}_k$ and $f_k$ are defined. Let $y^*=h_{k+1}$. Observe that $(y^*,z^*)\in \mathcal{Q}$. Hence, since (\ref{genvdi}) is valid for $\mathcal{Q}$, it follows that \begin{align*}
\beta&\leq y^*+\sum_{j\in [n]}\alpha_jz^*_j\\
&= h_{k+1}+\sum_{j\leq k}\alpha_j+\sum_{j> k}\alpha_jz^*_j\\
&= h_{k+1}+\sum_{j\leq k}\alpha_j+f_k(\alpha). \end{align*} Since this holds for all $0\leq k\leq \nu$, it follows that $(\alpha,\beta)\in \mathcal{G}$.

Conversely, suppose that $(\alpha,\beta)\in \mathcal{G}$. Let $(y^*,z^*)\in \mathcal{Q}$. Observe that $y^*\geq h_{\nu+1}$ by definition. Temporarily let $h_0:=+\infty$. Suppose that $h_k> y^*\geq h_{k+1}$ for some $0\leq k\leq \nu$. Note that $z^*_j=1$ for all $j\leq k$ and so $z^*\in \mathcal{P}_k$. Then \begin{align*}
y^*+\sum_{j\in [n]}\alpha_jz^*_j&\geq h_{k+1}+\sum_{j\leq k}\alpha_j+ \sum_{j>k}\alpha_jz^*_j\\
&\geq h_{k+1}+\sum_{j\leq k}\alpha_j+ f_k(\alpha)\\
&\geq \beta~~\text{ since } (\alpha,\beta)\in \mathcal{G} \end{align*} and so (\ref{genvdi}) is valid for $(y^*,z^*)$. Therefore, (\ref{genvdi}) is a valid inequality for $\mathcal{Q}$ and hence for $\text{conv}(\mathcal{Q})$, as claimed. \end{proof}

The above characterization theorem can be viewed as a generalization of a similar characterization given in Luedtke et. al \cite{luedtke2010}, which is only applicable to the case when the knapsack constraint $\sum_{j\in [n]}a_jz_j\leq p$ is a cardinality constraint. For the sake of completeness, here we will state the result, which is a slightly modified version of Theorem 3 in \cite{luedtke2010}.

\begin{thm}[\cite{luedtke2010}]\label{jimcharQ} Suppose that $a_j=1$ for all $j\in [n]$. Then any valid inequality for $\emph{conv}(\mathcal{Q})$ with nonzero coefficient on $y$ can be written in the form \begin{eqnarray}
y+\sum_{j\in [n]}\alpha_jz_j\geq \beta. \label{jimchar} \end{eqnarray} Furthermore, $(\ref{jimchar})$ is valid for $\emph{conv}(\mathcal{Q})$ if and only if \begin{eqnarray} \sum_{j\leq k-1}\alpha_j+\min_{S\in \mathcal{S}_k} \sum_{j\in S}\alpha_j+h_k\geq \beta~~\forall~ 1\leq k\leq p+1,\label{constraints}\end{eqnarray} where $\mathcal{S}_k:=\{S\subset \{k,\ldots,n\}:|S|\leq p-k+1\}$. \end{thm}

Observe that in the case when $a_j=1$ for all $j\in [n]$, we have that $\nu=p$ and $$f_{k-1}(\alpha)=\min_{S\in \mathcal{S}_k} \sum_{j\in S}\alpha_j.$$ As a result, the inequalities (\ref{constraints}) are equivalent to $$\sum_{j\leq k}\alpha_j+f_k(\alpha)+h_{k+1}\geq \beta~~\forall ~0\leq k\leq \nu,$$ which is equivalent to $(\alpha,\beta)\in \mathcal{G}$. Hence, Theorem \ref{jimcharQ} follows from Theorem \ref{validthm} and Lemma \ref{genform}.

Using Theorem \ref{validthm}, we can find cutting planes valid for conv$(\mathcal{Q})$ by solving a linear program.

\begin{thm}\label{septhm} Suppose $(y^*,z^*)\in \mathbb{R}_+\times \mathbb{R}^n_+$ satisfies $z^*\in \emph{conv}(\mathcal{P})$. Let \begin{eqnarray} LP^*:=\min\left\{\sum_{j\in [n]}\alpha_jz^*_j-\beta: (\alpha,\beta)\in \mathcal{G}\right\}.\label{sepopt}\end{eqnarray} Then $(y^*,z^*)\in \emph{conv}(\mathcal{Q})$ if and only if $y^*+LP^*\geq 0$. Also, if $y^*+LP^*<0$ and $(\alpha^*,\beta^*)$ is an optimal solution to \emph{(\ref{sepopt})}, then $y+\sum_{j\in [n]}\alpha^*_jz_j\geq \beta^*$ is a valid inequality for $\emph{conv}(\mathcal{Q})$ which is violated by $(y^*,z^*)$. \end{thm}

\begin{proof} If $(y^*,z^*)\in \text{conv}(\mathcal{Q})$, then $y^*+LP^*\geq 0$ by Theorem \ref{validthm}. Conversely, if $y^*+LP^*\geq 0$ then $$y^*+\sum_{j\in [n]}\alpha_jz^*_j\geq \beta$$ for all $(\alpha,\beta)\in \mathcal{G}$. Hence, by Theorem \ref{validthm}, $(y^*,z^*)$ satisfies all inequalities of the type (\ref{gvdi}) with $\gamma>0$. Moreover, since $z^*\in \text{conv}(\mathcal{Q})$, $(y^*,z^*)$ also satisfies all inequalities of the type (\ref{gvdi}) with $\gamma=0$. Hence, there is no valid inequality for $\text{conv}(\mathcal{Q})$ that separates $(y^*,z^*)$ from $\text{conv}(\mathcal{Q})$, so $(y^*,z^*)\in \text{conv}(\mathcal{Q})$. \end{proof}

Observe that the running time of solving the linear program (\ref{sepopt}) is polynomial in $n$ and $T_{\mathcal{P}}$, where $T_{\mathcal{P}}$ is the running time of optimizing $f_k$ over $\mathcal{P}_k$, for all $0\leq k \leq \nu$. Hence, solving (\ref{sepopt}) is efficient if minimizing over $\mathcal{P}_k$ can be accomplished efficiently, for all $0\leq k\leq \nu$.

\section{Facet-Defining Inequalities of conv($\mathcal{Q}$)}\label{fdisection}

The following theorem describes some interesting properties for facet-defining inequalities of $\text{conv}(\mathcal{Q})$. This helps us understand better the structure of the polyhedron conv($\mathcal{Q}$).

\begin{thm}\label{genfdithm} Suppose that the inequality \begin{eqnarray} y+\sum_{j\in [n]}\alpha_jz_j\geq \beta\label{genfdi}\end{eqnarray} is facet-defining for \emph{conv}$(\mathcal{Q})$ for some $(\alpha, \beta)\in \mathbb{R}^n\times \mathbb{R}$. Then \begin{enumerate}

\item[(i)] $(\alpha,\beta)$ is an extreme point of $\mathcal{G}$,

\item[(ii)] $\beta=h_1+f_0(\alpha)$, and

\item[(iii)] if $\alpha_k<0$ for some $1\leq k\leq n$, then $a_k>0$.

%\item[(iv)] there exists $k\in [n]$ such that $h_k=h_1$ and $\alpha_k\neq 0$. 
\end{enumerate} \end{thm}

\begin{proof} (i) This is true since otherwise (\ref{genfdi}) would be the convex combination of two distinct valid inequalities of conv($\mathcal{Q}$), which cannot be the case since (\ref{genfdi}) is a facet-defining inequality for conv($\mathcal{Q}$).

\noindent (ii) Let $z^*\in \mathcal{P}_0$ be a solution that attains the minimum of $f_0$. It is then by definition clear that $(h_1, z^*)\in \mathcal{Q}$. Thus by (\ref{genfdi}) we get that $$\beta\leq h_1+\sum_{j\in [n]}\alpha_jz^*_j=h_1+f_0(\alpha)$$ and so $\beta\leq h_1+f_0(\alpha)$. Since conv$(\mathcal{Q})$ is full-dimensional and (\ref{genfdi}) is a facet-defining inequality different from $z_1\leq 1$, it follows that there is a point $(y',z')\in \mathcal{Q}$ on the facet defined by (\ref{genfdi}) such that $z'_1=0$. Note that this implies that $y'\geq h_1$. Note also that $z'\in \mathcal{P}_0$. We thus have $$\beta=y'+\sum_{j\in [n]}\alpha_jz'_j\geq h_1+f_0(\alpha)$$ which implies that $\beta\geq h_1+f_0(\alpha)$. Hence, $\beta=h_1+f_0(\alpha)$, as claimed.

\noindent (iii) Suppose not, and assume that $a_k=0$ for some $1\leq k\leq n$ with $\alpha_k<0$. Since $a_k=0$ it follows that $$\text{proj}_{\mathbb{R}\times \mathbb{R}^{[n]\setminus \{k\}}}\mathcal{Q}=\left\{(y,z)\in \mathbb{R}_+\times \{0,1\}^{[n]\setminus \{k\}}:\sum_{j\in [n]\setminus \{k\}}a_jz_j\leq p,~y+h_iz_i\geq h_i~\forall i\in [n]\setminus \{k\}\right\}.$$ We claim that \begin{eqnarray} y+\sum_{j\in [n]\setminus \{k\}}\alpha_jz_j\geq \beta-\alpha_k\label{projvalid}\end{eqnarray} is a valid inequality for $\text{proj}_{\mathbb{R}\times \mathbb{R}^{[n]\setminus \{k\}}}\mathcal{Q}$. Let $(y^*,z^k)\in \text{proj}_{\mathbb{R}\times \mathbb{R}^{[n]\setminus \{k\}}}\mathcal{Q}$. Define $z^*\in \mathbb{R}^n$ as follows: $z^*_j=z^k_j$ if $j\neq k$ and $z^*_k=1$. Since $a_k=0$ it follows that $(y^*,z^*)\in \mathcal{Q}$. Thus, since (\ref{genfdi}) is valid for $\mathcal{Q}$, it follows that $$\beta-\alpha_k\leq y^*+\sum_{j\in [n]}\alpha_jz^*_j-\alpha_k= y^*+\sum_{j\in [n]\setminus \{k\}}\alpha_jz^*_j,$$ and so (\ref{projvalid}) is valid for $\text{proj}_{\mathbb{R}\times \mathbb{R}^{[n]\setminus \{k\}}}\mathcal{Q}$, and so it is also a valid inequality for $\mathcal{Q}$. However, the facet-defining inequality (\ref{genfdi}) is the sum of (\ref{projvalid}) and the inequality $-\alpha_k(1-z_k)\geq 0$, which is valid for $\mathcal{Q}$ since $\alpha_k<0$. But then (\ref{genfdi}) is dominated by (\ref{projvalid}), a contradiction. Thus, $a_k> 0$. %\noindent (iv) Suppose not. Let $K$ be the set of all $j\in [n]$ such that $h_k=h_1$. So $\alpha_j=0$ for all $j\in K$. $\cdots$ 
\end{proof}

In the next section, we give an explicit class of facet-defining inequalities of $\text{conv}(\mathcal{Q})$, which may be useful for generating cutting planes in a branch-and-cut algorithm.

\section{An Explicit Class of Facet-Defining Inequalities for $\text{conv}(\mathcal{Q})$}\label{explicitfdisub}

In this section, we will first give an overview of the known explicit class of facet-defining inequalities for conv($\mathcal{Q}$). We will then introduce a new explicit class of facet-defining inequalities that subsumes all the previously known classes.

Chronologically speaking, the first class of facet-defining inequalities for conv$(\mathcal{Q})$ are called the {\it strengthened star inequalities} (see \cite{luedtke2010, atamturk2000}).

\begin{thm}[\cite{luedtke2010, atamturk2000}]\label{starineqthm} The strengthened star inequalities \begin{eqnarray} y+\sum_{j=1}^{a}(h_{t_j}-h_{t_{j+1}})z_{t_j}\geq h_{t_1}~~~~\forall~T=\{t_1,\ldots,t_a\}\subset \{1,\ldots,\nu\} \label{starineq}\end{eqnarray} with $t_1<\cdots<t_a$ and $h_{t_{a+1}}:=h_{\nu+1}$ are valid for $\emph{conv}(\mathcal{Q})$. Moreover, $(\ref{starineq})$ is facet-defining for $\emph{conv}(\mathcal{Q})$ if and only if $h_{t_1}=h_1$. \end{thm}

As shown in \cite{gunluk2001, atamturk2000, guan2007}, the (strengthened) star inequalities can be separated in polynomial time and are sufficient to describe the convex hull of $$\mathcal{R}:=\left\{(y,z)\in \mathbb{R}_+ \times \{0,1\}^n:y+h_iz_i\geq h_i ~ \forall i\in [n]\right\}.$$ However, as it turns out, when a knapsack constraint is enforced in $\mathcal{R}$ to obtain $\mathcal{Q}$, the convex hull becomes much more complex.

%(EXAMPLE)

Later, Luedtke et. al \cite{luedtke2010} found a larger and subsuming class of facet-defining inequalities for conv$(\mathcal{Q})$ in the case when the knapsack constraint $\sum_{j\in [n]}a_jz_j\leq p$ is just a cardinality constraint, i.e. $a_j=1$ for all $j\in [n]$. Subsequently, K\"{u}\c{c}\"{u}kyavuz \cite{kucukyavuz2012} introduced 
an even larger and subsuming class facet-defining inequalities for conv$(\mathcal{Q})$, called the $(T,\Pi_L)$ {\it inequalities}, which again only applies to the case when the knapsack constraint $\sum_{j\in [n]}a_jz_j\leq p$ is just a cardinality constraint. Here, we only state the latter class.

\begin{thm}[\cite{kucukyavuz2012}] Suppose that $a_j=1$ for all $j\in [n]$. Take a positive integer $m\leq \nu=p$. Suppose that \begin{enumerate}

\item[(i)] $T:=\{t_1,\ldots,t_a\}\subset \{1,\ldots,m\}$, where $t_1<\ldots <t_a$; and

\item[(ii)] $L\subset \{m+2,\ldots,n\}$ and take a permutation of the elements in $L$, $\Pi_L:=\{\ell_1,\ldots,\ell_{p-m}\}$ such that $\ell_j> m+j$ for all $1\leq j\leq p-m$. \end{enumerate} Set $t_{a+1}:=m+1$. Let $\Delta_1:=h_{m+1}-h_{m+2}$, and for $2\leq j\leq p-m$ define $$\Delta_j:= \max\left\{\Delta_{j-1}, h_{m+1}- h_{m+1+j}-\sum{\left(\Delta_i:\ell_i> m+j,i<j\right)}\right\}.$$
Then the $(T,\Pi_L)$ inequality \begin{eqnarray}
y+\sum_{j=1}^{a}(h_{t_j}-h_{t_{j+1}}) z_{t_j}+\sum_{j=1}^{p-m}{\Delta_j(1-z_{q_j})} \geq h_{t_1} \label{TPiineq}\end{eqnarray} is valid for $\emph{conv}(\mathcal{Q})$. Furthermore, $(\ref{TPiineq})$ is facet-defining inequality for $\emph{conv}(\mathcal{Q})$ if and only if $h_{t_1}=h_1$. \end{thm}

Observe that the $(T,\emptyset)$ inequalities are simply the strengthened star inequalities.

%(EXAMPLE)

We now introduce a larger and subsuming class of facet-defining inequalities for conv$(\mathcal{G})$ in the general setting that coincides with the $(T,\Pi_L)$ inequalities in the case when $a_j=1$ for all $j\in [n]$. For all $1\leq m\leq n$, let $$s_m:=\sum_{j=1}^ma_j,$$ and let $s_0:=0$.

\begin{thm}\label{explicitfdi} Take an integer $0\leq m\leq \nu$ such that $p-s_m$ is an integer. For each $1\leq j\leq p-s_m$, let $k(j):=\max\{k:j\geq s_k-s_m\}$. Let \begin{enumerate}

\item[(i)] $T:=\{t_1,\ldots,t_a\}\subset \{1,\ldots,m\}$ where $t_1<\ldots <t_a$;

\item[(ii)] $S:=\{q_1,\ldots,q_s\}\subset \{m+2,\ldots,n\}$ where $s=p-s_m$ and $q_j> k(j)$ for all $1\leq j \leq p-s_m$; and

\item[(iii)] $S$ is chosen so that $a_j=1$ for all $j\in S$, and $a_j\leq s_m$ for all $j\notin S$. \end{enumerate}  Set $t_{a+1}=m+1$. Let $\alpha_{q_1}:=h_{k(1)+1}-h_{m+1}$, and for $2\leq j\leq p-s_m$, define \begin{eqnarray}\alpha_{q_j}:= \min\left\{\alpha_{q_{j-1}},  h_{k(j)+1}-h_{m+1}-\sum{\left(\alpha_{q_i}:q_i> k(j),i<j\right)}\right\}.\label{coeff}\end{eqnarray} Then
\begin{eqnarray} y+\sum_{j=1}^{a}(h_{t_j}-h_{t_{j+1}}) z_{t_j}+\sum_{i\in S}\alpha_iz_i \geq h_{t_1}+\sum_{i\in S}\alpha_i \label{fdieg1} \end{eqnarray} is a valid inequality for \emph{conv}$(\mathcal{Q})$. Furthermore, $(\ref{fdieg1})$ is a facet-defining inequality for \emph{conv}$(\mathcal{Q})$ if and only if $h_{t_1}=h_1$. \end{thm}

We would like to explain how this theorem implies that the $(T,\Pi_L)$ inequalities (\ref{TPiineq}) are facet-defining for conv$(\mathcal{Q})$ when $a_j=1$ for all $j\in [n]$. Let $S:=L$ and $q_i:=\ell_i$ for all $1\leq i\leq p-m$. Observe that $s_m=m$ is an integer, $|S|=p-m=p-s_m$, and $k(j)=m+j$ for all $1\leq j\leq p-s_m$. Note that $q_j=\ell_j>m+j=k(j)$ for all $1\leq j\leq p-s_m$. Also, note that $a_k=1\leq s_m$ for all $k\in [n]\setminus (T\cup S)$ since $m\geq 1$. Lastly, observe that $\alpha_j=-\Delta_j$ for all $j\in S$. Hence, by Theorem \ref{explicitfdi}, the $(T,\Pi_L)$ inequalities (\ref{TPiineq}) are facet-defining for conv$(\mathcal{Q})$ when $a_j=1$ for all $j\in [n]$.

Observe that Theorem \ref{explicitfdi} can also be applied to any scalar multiple of the knapsack constraint, and this will potentially give us more facet-defining inequalities. That is, one can apply Theorem \ref{explicitfdi} to $\sum_{j\in [n]}da_jz_j\leq dp$, for any arbitrary positive real number $d$.

%In Sect. \ref{fdiexamples}, we give a few examples explaining the newly introduced facet-defining inequalities. 
In Sect. \ref{sep}, we give a polynomial time separation algorithm to separate over a subset of (\ref{fdieg1}). This separation algorithm is analogous to and an appropriate generalization of the one given in K\"{u}\c{c}\"{u}kyavuz \cite{kucukyavuz2012}. Finally in Sect. \ref{prooffdi}, we give a proof of Theorem \ref{explicitfdi}.

%\subsection{Examples}\label{fdiexamples}

%$\vdots$

\subsection{Separation of the new class of FDIs (\ref{fdieg1})}\label{sep}

In this section, we give a polynomial time separation algorithm over a subset of the inequalities (\ref{fdieg1}). This algorithm is analogous to the separation algorithm given in \cite{kucukyavuz2012}.

\begin{thm}\label{sepoverfdis} Take $0\leq m\leq \nu$ and $0\leq r\leq p-s_m$. Let $A_m:=\{m+2\leq j\leq n: a_j=1\}$. Suppose that $a_j\leq s_m$ for all $j\in [n]$, $k(1)<k(2)<\cdots<k(r)<k(r+1)$ and $F:=\{k(1)+1,\ldots,k(r)+1\}\subset A_m$. Then we can find the most violated inequality $(\ref{fdieg1})$ with $m$ as above and $S=F\cup G$ with $G\subset \{k(p-s_m)+1,\ldots,n\}$ in $O(p^3)$. \end{thm}

\begin{proof} Suppose that $m$ and $r$ are given as above, and that $S=F\cup G$ with $G\subset \{k(p-s_m)+1,\ldots,n\}$. With this choice of $S$, we must have that $q_j=k(j)+1$ for all $1\leq j\leq r$ (note (ii)). As a result, $\alpha_{q_j}$ in (\ref{coeff}) simplifies to $\alpha_{q_j}=\min\{\alpha_{q_{j-1}}, h_{k(j)+1}-h_{m+1}\}$ for $2\leq j\leq r$. Moreover, for $q_i\in G\subset \{k(p-s_m)+1,\ldots,n\}$, we have $q_i> k(p-s_m)\geq k(j)$ for all $r+1\leq j\leq p-s_m$. Hence, $\alpha_{q_j}=\min\{\alpha_{q_{j-1}}, h_{k(j)+1}-h_{m+1}-\sum_{i=r+1}^{j-1}\alpha_{q_i}\}$ for $r+1\leq j\leq p-s_m$. Observe that, assuming $S=F\cup G$, the coefficients $\alpha_{q_j}$ do not depend on a particular choice of $G$, but depend only on $\alpha_r$.

Let $(y^*,z^*)\in \mathbb{R}_+\times \mathbb{R}^n_+$. We now give an algorithm to to identify the most violated inequality (\ref{fdieg1}) with $S=F\cup G$ and $G\subset \{k(p-s_m)+1,\ldots,n\}$. Note that the problem of finding the best $T$ in inequalities (\ref{fdieg1}) can be solved as a shortest path problem on a directed acyclic graph with $O(p^2)$ arcs. For details, see \cite{kucukyavuz2012}. Note that we always include a $t_i$ in $T$ for which $h_{t_i}=h_1$.

To find the set $G$ that gives the most violated inequality (\ref{fdieg1}) in the desired form, we keep an ordered list of the elements in $\{k(p-s_m)+1,\ldots,n\}$, denoted by $Z$, in decreasing order of $z^*_j$ for $k(p-s_m)+1\leq j\leq n$ and we choose the first $p-s_m-r$ elements in $Z$ to be in the set $G$.

As a result, the most violated inequality $(\ref{fdieg1})$ with $m$ as above and $S=F\cup G$ with $G\subset \{k(p-s_m)+1,\ldots,n\}$ in $O(p^3)$. \end{proof}

\begin{crly}\label{sepoverfdisgen} Take $0\leq m\leq \nu$. Let $A_m:=\{m+2\leq j\leq n: a_j=1\}$. Suppose that $a_j\leq s_m$ for all $j\in [n]$, and $k(1)<k(2)<\cdots<k(p-s_m)$. Then we can find the most violated inequality $(\ref{fdieg1})$ with $m$ as above and $S=\{k(1)+1,\ldots,k(r)+1\}\cup G$ with $G\subset \{k(p-s_m)+1,\ldots,n\}$ over all $0\leq r\leq p-s_m$ in $O(p^4)$. \end{crly}

\subsection{Proof of Theorem \ref{explicitfdi}}\label{prooffdi}

As we will see, the function $f_k$ plays a central role in the proof of Theorem \ref{explicitfdi}. In the following lemma, which is needed for the proof, we compute $f_k$ for each $k$ under the nice assumptions on $S$ given in Theorem \ref{explicitfdi}. Recall that $s_k=\sum_{j=1}^ka_j$ for all $1\leq k\leq n$, and $s_0=0$.

\begin{lma} Let $R\cup S$ be a partition of $[n]$, $\alpha\in \mathbb{R}^n$ and $0\leq m\leq \nu$ such that the following are satisfied: \begin{enumerate}

\item[(i)] $p-s_m$ is an integer, $|S|=p-s_m$ and $a_j=1$ for all $j\in S$;

\item[(ii)] if $i\in R$ then $\alpha_i\geq 0$, and if $i\in S$ then $\alpha_i\leq 0$;

\item[(iii)] $S\subset \{m+2,\ldots, n\}$ and $\alpha_{q_1}\geq \cdots\geq \alpha_{q_{|S|}}$ for some permutation $(q_1,\ldots,q_{|S|})$ on $S$;

\item[(iv)] $i> s_k-s_m$ implies that $q_i>k$ for all $i\in \{1,\ldots,p-s_m\}$ and $k\in \{m+1,\ldots,\nu\}$.\end{enumerate} Then $$f_k(\alpha)=\left\{\begin{array}{ll} \sum_{j\in S}\alpha_j ~~~~~~~~~~~~~~~~~~ \text{ if } 0\leq k\leq m; \\ \sum{(\alpha_{q_i}:i> s_k-s_m)} ~~~\text{ if } m+1\leq k\leq \nu.\end{array} \right.$$\end{lma}

\begin{proof} If $0\leq k\leq m$, then $z^*\in \{0,1\}^{[n]}$ defined as $$z^*_j:=\left\{\begin{array}{ll} 1 ~~~ \text{ if } j\in S; \\ 0 ~~~ \text{ otherwise.} \end{array} \right.$$ is a feasible point for $\mathcal{P}_k$ since $$\sum_{j>k}a_jz^*_j=\sum_{j\in S,j>k}z^*_j=|S|= p-s_m\leq p-\sum_{j\leq k}a_j.$$ Hence, $$\sum_{j\in S}\alpha_j\leq f_k(\alpha)\leq \sum_{j\in S}\alpha_j$$ and so $f_k(\alpha)=\sum_{j\in S}\alpha_j$. Now choose $m+1\leq k\leq \nu$. We will first find a lower bound for $f_k(\alpha)$. Let $z\in \mathcal{P}_k$. Then \begin{align*}
|\{j\in S:j> k,z_j=1\}|=\sum_{j\in S,j>k}z_j=\sum_{j\in S,j>k}a_jz_j \leq p-\sum_{j\leq k}a_j &=p-s_k\\
&=|S|-s_k+s_m\\
&= |\{q_i\in S:i>s_k-s_m\}|.\end{align*} Observe that $i\geq s_k-s_m$ implies that $q_i>k$. As a result, since $\alpha_{q_1}\geq \cdots\geq \alpha_{q_{|S|}}$, we get that $$\sum_{j\in S,j>k}\alpha_jz_j\geq \sum{\left(\alpha_{q_i}:i> s_k-s_m\right)}.$$ Since this is true for all $z\in \mathcal{P}_k$, it follows that $$f_k(\alpha)\geq \sum{\left(\alpha_{q_i}:i> s_k-s_m\right)}.$$ Furthermore, we claim that equality holds above. Define $z^*\in \{0,1\}^{[n]}$ as follows: for $q_i\in S$ let $$z^*_{q_i}:=\left\{\begin{array}{ll} 1 ~~~ \text{ if } i> s_k-s_m \text{ or } q_i\leq k; \\ 0 ~~~ \text{ otherwise,} \end{array} \right.$$ and for $i\in R$ let $z^*_i:=0$. We have \begin{align*}
\sum_{i>k}a_iz^*_i=\sum_{q_i>k}z^*_{q_i}
&=\sum (z^*_{q_i}:q_i>k, i>s_k-s_m)\\
&=\sum (z^*_{q_i}: i>s_k-s_m)~\text{ since } i>s_k-s_m \text{ implies } q_i>k\\
&= |S|-(s_k-s_m)\\
&=p-s_k=p-\sum_{j\leq k}a_j.\end{align*} Hence, $z^*\in \mathcal{P}_k$. However, $$f_k(\alpha)\leq \sum_{j\in S,j>k}\alpha_jz^*_j=\sum{\left(\alpha_{q_i}:i> s_k-s_m\right)}\leq f_k(\alpha)$$ and so $$f_k(\alpha)= \sum{\left(\alpha_{q_i}:i> s_k-s_m\right)}.$$ Hence, we are done. \end{proof}

Now we are ready to prove Theorem \ref{explicitfdi}. We restate Theorem \ref{explicitfdi} for convenience.

\noindent {\bf Restatement of Theorem \ref{explicitfdi}.} Take an integer $0\leq m\leq \nu$ such that $p-s_m$ is an integer. For each $1\leq j\leq p-s_m$, let $k(j):=\max\{k:j\geq s_k-s_m\}$. Let \begin{enumerate}

\item[(i)] $T:=\{t_1,\ldots,t_a\}\subset \{1,\ldots,m\}$ where $t_1<\ldots <t_a$;

\item[(ii)] $S:=\{q_1,\ldots,q_s\}\subset \{m+2,\ldots,n\}$ where $s=p-s_m$ and $q_j> k(j)$ for all $1\leq j \leq p-s_m$; and

\item[(iii)] $S$ is chosen so that $a_j=1$ for all $j\in S$, and $a_j\leq s_m$ for all $j\notin S$. \end{enumerate}  Set $t_{a+1}=m+1$. Let $\alpha_{q_1}:=h_{k(1)+1}-h_{m+1}$, and for $2\leq j\leq p-s_m$, define $$\alpha_{q_j}:= \min\left\{\alpha_{q_{j-1}},  h_{k(j)+1}-h_{m+1}-\sum{\left(\alpha_{q_i}:q_i> k(j),i<j\right)}\right\}.$$ Then
\begin{eqnarray} y+\sum_{j=1}^{a}(h_{t_j}-h_{t_{j+1}}) z_{t_j}+\sum_{i\in S}\alpha_iz_i \geq h_{t_1}+\sum_{i\in S}\alpha_i \label{fdieg} \end{eqnarray} is a valid inequality for conv$(\mathcal{Q})$. Furthermore, $(\ref{fdieg})$ is a facet-defining inequality for conv$(\mathcal{Q})$ if and only if $h_{t_1}=h_1$.

\begin{proof} Let $R=[n]\setminus S$ and define, for $i\in R$, $$\alpha_i:=\left\{\begin{array}{ll} h_{t_j}-h_{t_{j+1}} ~~~ \text{ if } i=t_j \text{ for some } 1\leq j\leq m; \\ 0 ~~~~~~~~~~~~~~~ \text{ otherwise.} \end{array} \right.$$ We first show that $(\alpha, h_{t_1}+\sum_{i\in S}\Delta_i)\in \mathcal{G}$. Observe that if $i\in R$ then $\alpha_i\geq 0$, and if $i\in S$ then $\alpha_i\leq 0$. Also, note that $\alpha_{q_1}\geq \cdots\geq \alpha_{q_{|S|}}$. Let $0\leq k\leq \nu$. If $0\leq k\leq m$, then by the previous lemma, $f_k(\alpha)=\sum_{i\in S}\alpha_i$. Suppose that $t_j< k+1\leq t_{j+1}$ for some $j\in \{0,1,\ldots,a\}$ where $t_0:=0$. Then $$\sum_{i\leq k}\alpha_i+f_k(\alpha)-h_{t_1}-\sum_{i\in S}\alpha_i= \sum_{i=1}^{j}(h_{t_i}-h_{t_{i+1}})-h_{t_1}=h_{t_1}-h_{t_{j+1}}-h_{t_1}\geq -h_{k+1}.$$ Otherwise, assume that $m+1\leq k\leq \nu$. Note that in this case we have $$\sum_{i\in R,i\leq k}\alpha_i=\sum_{i\in T}\alpha_i=h_{t_1}-h_{m+1}.$$ Also, by the previous lemma, we have $$f_k(\alpha)= \sum{\left(\alpha_{q_i}:i> s_k-s_m\right)}.$$ Then \begin{align*}
&\sum_{i\in R,i\leq k}\alpha_i+\sum_{i\in S,i\leq k}\alpha_i+f_k(\alpha)-h_1-\sum_{i\in S}\alpha_i\\
&=h_{t_1}-h_{m+1}+\sum{\left(\alpha_{q_i}:q_i\leq k\right)} +\sum{\left(\alpha_{q_i}:i> s_k-s_m\right)}-h_{t_1}-\sum_{i\in S}\alpha_i\\
&=-h_{m+1}-\alpha_{q_j}- \sum{\left(\alpha_{q_i}:q_i>k,i<j\right)} ~~\text{ for } j=s_k-s_m\\
&\geq -h_{k+1}.\end{align*} As a result, $(\alpha, h_{t_1}+\sum_{i\in S}\alpha_i)\in \mathcal{G}$ and so by Theorem \ref{validthm}, we get that (\ref{fdieg}) is a valid inequality for conv$(\mathcal{Q})$.

Observe that if (\ref{fdieg}) is facet-defining for conv$(\mathcal{Q})$, then by Theorem \ref{genfdithm} (ii), we must have that $h_{t_1}+\sum_{i\in S}\alpha_i=h_1+f_0(\alpha)$. However, by the previous lemma, we know that $f _0(\alpha)=\sum_{i\in S}\alpha_i$. Hence, $h_{t_1}=h_1$ is a necessary condition for (\ref{fdieg}) to be facet-defining. Conversely, assume that $h_{t_1}=h_1$. We will find $n+1$ affinely independent points in $\mathcal{Q}$ that satisfy (\ref{fdieg}) at equality.

For each $k:=t_j\in T$, let $y^k=h_{t_j}$ and define $z^k\in \{0,1\}^{[n]}$ as follows: $$z^k_i:=\left\{\begin{array}{ll} 1 ~~~ \text{ if } i< k \text{ or } i\in S; \\ 0 ~~~ \text{ otherwise.} \end{array} \right.$$ Note that $z^k_i=1$ for all $i<k$. Also, we have $$\sum_{i=1}^{n}a_iz^k_i= \sum_{i<k}a_iz^k_i+ \sum_{i\in S}z^k_i=s_{k-1}+|S|=s_{k-1}+p-s_m\leq p.$$ Hence, $(y^k,z^k)\in \mathcal{Q}$. Moreover, \begin{align*}
y^k+\sum_{i=1}^{a}(h_{t_i}-h_{t_{i+1}}) z^k_{t_i}+\sum_{i\in S}\alpha_iz^k_i=h_{t_j}+\sum_{t_i<t_j}(h_{t_i}-h_{t_{i+1}})+\sum_{i\in S}\alpha_i&=h_{t_j}+h_1-h_{t_j}+\sum_{i\in S}\alpha_i\\
&=h_1+\sum_{i\in S}\alpha_i.\end{align*}

For each $k:=q_j\in S$, define $$\ell(j):=\max\left\{1\leq \ell\leq j:\alpha_{q_j}=h_{k(\ell)+1}-h_{m+1}-\sum{\left(\alpha_{q_i}:q_i> k(\ell),i<\ell\right)} \right\}.$$ Now let $y^k:=h_{k(\ell(j))+1}$ and define $z^k\in \{0,1\}^{[n]}$ as follows: $$z^k_t:=\left\{\begin{array}{ll}  0 ~~~ \text{ if } t=q_i \text{ and } q_i>k(\ell(j)) \text{ and } i<\ell(j), \text{ or } t=q_j, \text{ or } t\in R \text{ and } t>k(\ell(j));\\ 1 ~~~ \text{ otherwise.} \end{array} \right.$$ Note that $z^k_i=1$ for all $i\leq k(\ell(j))$. Also, we have \begin{align*}
\sum_{i=1}^{n}a_iz^k_i&= \sum_{i\leq k(\ell(j))}a_i+
|\{q_i\in S:q_i>k(\ell(j)), i\geq \ell(j), i\neq j \}|\\
&=s_{k(\ell(j))}+|\{q_i\in S:q_i>k(\ell(j)), i\geq \ell(j) \}|-1\\
&=s_{k(\ell(j))}+|\{q_i\in S:i\geq \ell(j) \}|-1\\
&=s_{k(\ell(j))}+|S|-\ell(j)\\
&=s_{k(\ell(j))}+p-s_m-\ell(j)\\
&\leq p ~~\text{ by definition of } k(\cdot).\end{align*} Hence, $(y^k,z^k)\in \mathcal{Q}$. Moreover, \begin{align*}
& y^k+\sum_{i=1}^{a}(h_{t_i}-h_{t_{i+1}}) z^k_{t_i}+\sum_{i=1}^{s}\alpha_{q_i}z^k_{q_i}\\
&=h_{k(\ell(j))+1}+ h_1-h_{m+1}+\sum_{i\in S}\alpha_i-\Delta_{q_j}-\sum {\left(\alpha_{q_i}:q_i> k(\ell(j)),i<\ell(j)\right)} \\
&= h_{k(\ell(j))+1}+ h_1-h_{m+1}+\sum_{i\in S}\alpha_i+h_{m+1}-h_{k(\ell(j))+1}\\
&=h_1+\sum_{i\in S}\alpha_i.\end{align*}

For all $k\in R\setminus T$, let $y^k:=h_1$ and define $z^k\in \{0,1\}^{[n]}$ as follows: $$z^k_i:=\left\{\begin{array}{ll}  0 ~~~ \text{ if } i\in R\setminus \{k\};\\ 1 ~~~ \text{ otherwise.} \end{array} \right.$$ We have $$\sum_{i=1}^{n}a_iz^k_i= a_k+|S|=a_k+p-s_m\leq p ~~~ \text{ by (iii).}$$  Hence, $(y^k,z^k)\in \mathcal{Q}$. Moreover, $$y^k+\sum_{i=1}^{a}(h_{t_i}-h_{t_{i+1}}) z^k_{t_i}+\sum_{i\in S}\alpha_iz^k_i=h_1+\sum_{i\in S}\alpha_i.$$

Lastly, let $y^0=h_{m+1}$ and define $z^0\in \{0,1\}^{[n]}$ as follows: $$z^0_i:=\left\{\begin{array}{ll} 1 ~~~ \text{ if } i< m+1 \text{ or } i\in S; \\ 0 ~~~ \text{ otherwise.} \end{array} \right.$$ Note that $z^0_i=1$ for all $i<m+1$. Also, we have $$\sum_{i=1}^{n}a_iz^k_i= \sum_{i<m+1}a_iz^k_i+ \sum_{i\in S}z^k_i=s_{m}+|S|=s_m+p-s_m= p.$$ Hence, $(y^0,z^0)\in \mathcal{Q}$. Moreover, \begin{align*}
y^0+\sum_{i=1}^{a}(h_{t_i}-h_{t_{i+1}}) z^0_{t_i}+\sum_{i\in S}\alpha_iz^0_i&=h_{m+1}+\sum_{i=1}^{a}(h_{t_i}-h_{t_{i+1}})+\sum_{i\in S}\alpha_i\\
&=h_{m+1}+h_1-h_{m+1}+\sum_{i\in S}\alpha_i\\
&=h_1+\sum_{i\in S}\alpha_i.\end{align*}

Hence, the face defined by (\ref{fdieg}) contains $z^0,z^1,\ldots,z^{n}$, which are $n+1$ affinely independent points in $\mathcal{Q}$. As a result, (\ref{fdieg}) is a facet-defining inequality for conv($\mathcal{Q}$).\end{proof}

\section{Heuristic Separation over $\text{conv}(\mathcal{Q})$}\label{heuristicsep}

In this section, we give a polynomial time algorithm that separates over a subset of inequalities of the type $$ y+\sum_{j\in [n]}\alpha_jz_j\geq \beta~~~~\forall~\alpha\in \mathbb{R}^n:\alpha_j\geq 0 \text{ if and only if } j\in R,$$ for a fixed subset $R$ of $[n]$. When $a_i=a_j$ for all $i,j\notin R$, this separation is exact.

Let $R\cup S$ be a partition of $[n]$. Define the polyhedron $$\mathcal{G}(R):=\left\{(\delta,\Delta,h)\in \mathbb{R}^R_+\times \mathbb{R}^S_+\times \mathbb{R}: (\ref{Gconst})\right\}$$ where \begin{align}
\sum_{j\in R,j\leq k}\delta_j+ \sum_{j\in S,j> k}\Delta_j(1-z^*_j) +h_{k+1}&\geq h ~~~~\forall z^*\in \mathcal{P}_k, \forall ~0\leq k\leq \nu.\label{Gconst} \end{align} The following theorem explains the importance of the polyhedron $\mathcal{G}(R)$.

\begin{thm}\label{RSvalidthm} Let $R\cup S$ be a partition of $[n]$. Choose $(\delta,\Delta,h)\in \mathbb{R}_+^R\times \mathbb{R}_+^S\times \mathbb{R}$. Then $$y+\sum_{j\in R}\delta_jz_j+\sum_{j\in S}\Delta_j(1-z_j)\geq h$$ is a valid inequality for $\emph{conv}(\mathcal{Q})$ if and only if $(\delta,\Delta,h)\in \mathcal{G}(R)$.\end{thm}

\begin{proof} The proof is very similar to that of Theorem \ref{validthm} and is therefore omitted. \end{proof}

Again, we would like to point out that $\mathcal{G}(R)$ in general has exponentially many constraints. However, $\mathcal{G}(R)$ has an alternate formulation with $O(n)$ non-linear constraints: for $\Delta\in \mathbb{R}_+^S$ and $k\in \mathbb{Z}$, let $$g_k(\Delta):=\min \left\{\sum_{j\in S,j> k}\Delta_j(1-z_j):z\in \mathcal{P}_k \right\}. $$ Then we have $$\mathcal{G}(R)=\left\{(\delta,\Delta,h)\in \mathbb{R}_+^R\times \mathbb{R}_+^S\times \mathbb{R}: \sum_{j\in R,j\leq k}\delta_j+g_k(\Delta)+h_{k+1} \geq h ~\forall ~0\leq k\leq \nu \right\}.$$ We will now find a polyhedron that is contained in $\mathcal{G}(R)$ and is equal to $\mathcal{G}(R)$ when $a_i=a_j$ for all $i,j\in S$, and can be described efficiently.

\begin{ntn} Let $\Delta=(\Delta_j:j\in S)\subset \mathbb{N}$ where $S$ is some index set, and let $k,l\in \mathbb{N}$. Define $\Delta [k,l]$ to be the sum of the smallest $l-|\{j\in S:j\leq k\}|$ elements in $\{\Delta_j: j> k\}$. \end{ntn}

Let $R\cup S$ be a partition of $[n]$, and suppose that $m_S:=\min\{a_j:j\in S\}>0$.

\begin{obs}\label{approxobs} Given $\Delta\in \mathbb{R}^S_+$ and $0\leq k\leq \nu$, we have $$g_k(\Delta)\geq \Delta\left[k,|S|-\left \lfloor\frac{p-s_k}{m_S}\right\rfloor\right].$$ Moreover, the inequality is tight when $a_j=m_S$ for all $j\in S$. \end{obs}

\begin{proof} Choose $z\in \mathcal{P}_k$. We have that $$|\{j\in S:j>k, z_j=1\}|= \sum_{j\in S, j>k}z_j\leq \frac{1}{m_S}\sum_{j\in S, j>k}a_jz_j\leq \frac{1}{m_S}\left(p-\sum_{j\leq k}a_j\right)=\frac{p-s_k}{m_S}.$$ Therefore, \begin{align*}
|\{j\in S:j> k, z_j=0\}|&=|\{j\in S:j>k\}|-|\{j\in S:j>k, z_j=1\}|\\
&\geq |S|-|\{j\in S:j\leq k\}|-\left\lfloor\frac{p-s_k}{m_k}\right\rfloor.\end{align*} As a result, we obtain that $\sum_{j\in S, j>k}\Delta_j(1-z_j)\geq \Delta\left[k,|S|-\left \lfloor\frac{p-s_k}{m_k}\right\rfloor\right].$ Since this is true for all $z\in \mathcal{P}_k$, it follows that $$g_k(\Delta)\geq \Delta\left[k,|S|-\left \lfloor\frac{p-s_k}{m_k}\right\rfloor\right],$$ as claimed. In the case when $a_j=m_S$ for all $j\in S$, consider the point $z^*\in \{0,1\}^n$ defined as follows: $z^*_j=1$ if $j\leq k$ or $\Delta_j$ corresponds to one of the largest $\left\lfloor \frac{p-s_k}{m_S}\right \rfloor$ elements in $\{\Delta_j:j\in S:j>k\}$, and $z^*_j=0$ otherwise. Then $z^*\in \mathcal{P}_k$ and $\sum_{j\in S, j>k}\Delta_j(1-z^*_j)=\Delta\left[k,|S|-\left \lfloor\frac{p-s_k}{m_k}\right\rfloor\right].$ Hence, when $a_j=m_S$ for all $j\in S$, we have that $g_k(\Delta)= \Delta\left[k,|S|-\left \lfloor\frac{p-s_k}{m_S}\right\rfloor\right].$\end{proof}

Define the polyhedron $$\mathcal{A}(R):=\left\{(\delta,\Delta,h)\in \mathbb{R}_+^R\times \mathbb{R}_+^S\times \mathbb{R}: (\ref{ARSconst})\right\}$$ where \begin{eqnarray}
\sum_{j\in R, j\leq k}\delta_j+ \Delta\left[k,|S|-\left \lfloor\frac{p-s_k}{m_S}\right\rfloor\right] +h_{k+1}\geq h ~~~~ \forall ~0\leq k\leq \nu.\label{ARSconst} \end{eqnarray} Note that Observation \ref{approxobs} implies that $\mathcal{A}(R)\subset \mathcal{G}(R)$, and that $\mathcal{A}(R)= \mathcal{G}(R)$ when $a_j=m_S$ for all $j\in S$.

For each $0\leq k\leq \nu$, let $\beta^k:=|S|-|\{j\in S:j\leq k\}|-\left\lfloor\frac{p-s_k}{m_S}\right\rfloor$. We have $$\begin{array}{ccl} \Delta\left[k,|S|-\left\lfloor\frac{p-s_k}{m_S}\right\rfloor\right] = &\min & \sum_{j\in S, j>k} \Delta_jx_j\\
&{\rm s.t.} & \sum_{j\in S,j>k}x_j=\beta^k, \\ & & x_j\in [0,1] ~~\forall ~j\in S,j>k.\end{array}$$ (Note that the constraint matrix of the above linear program is totally unimodular.) Hence, by LP duality, we obtain that $$\begin{array}{ccl} \Delta\left[k,|S|-\left\lfloor\frac{p-s_k}{m_S}\right\rfloor\right] = &\max & \beta^k\gamma^k+\sum_{j\in S,j>k} \rho^k_j\\
&{\rm s.t.} & \gamma^k+\rho^k_j\leq \Delta_j~~~~~~~\forall ~j\in S,j>k, \\ & & \rho^k_j\leq 0 ~~~~~~~~~~~~~~~~\forall ~j\in S,j>k.\end{array}$$ Now let $\rho^k_j=0$ for all $j\in S$ with $j\leq k$. So $\rho^k\in \mathbb{R}^S_-$ for all $0\leq k\leq \nu$. With this, define the polyhedron $$\Theta(R):=\left\{\left(\delta,\Delta,h, \left(\gamma^k\right),\left(\rho^k\right)\right)\in \mathbb{R}_+^R\times \mathbb{R}_+^S\times \mathbb{R} \times \mathbb{R}^{\nu+1}\times \mathbb{R}_-^{S\times {\nu+1}}: (\ref{thetaconst1})-(\ref{thetaconst3})\right\}$$ where \begin{align}
\sum_{j\in R, j\leq k}\delta_j+\beta^k\gamma^k+\sum_{j\in S} \rho^k_j+h_{k+1} &\geq h ~~~~~~~\forall ~0\leq k\leq \nu, \label{thetaconst1}\\
\gamma^k+\rho^k_j&\leq \Delta_j~~~~~~\forall ~j\in S,j>k, \forall ~0\leq k\leq \nu,\label{thetaconst2}\\
\rho^k_j&=0~~~~~~~~\forall~j\in S,j\leq k,\forall ~0\leq k\leq \nu.\label{thetaconst3} \end{align}

Hence, by weak and strong LP duality we get the following.

\begin{lma}\label{RSprojlemma} $\mathcal{A}(R)=\emph{proj}_{\mathbb{R}^R\times \mathbb{R}^S\times \mathbb{R}}\Theta(R)$. \end{lma}

This lemma implies the following theorem.

\begin{thm} Let $R\cup S$ be a partition of $[n]$ such that $m_S>0$. Let $(y^*,z^*)\in \mathbb{R}\times \mathbb{R}^n_+$ and \begin{eqnarray} LP^*:=\min\left\{\sum_{j\in R}\delta_jz^*_j+\sum_{j\in S}\Delta_j(1-z^*_j)- h:(\delta,\Delta,h)\in \mathcal{A}(R)\right\}.\label{RSsepopt}\end{eqnarray} If $y^*+LP^*<0$ and $(\delta^*,\Delta^*,h^*)$ is an optimal solution to $(\ref{RSsepopt})$, then $(y^*,z^*)\notin \emph{conv}(\mathcal{Q})$ and $y+\sum_{j\in R}\delta^*_jz_j+\sum_{j\in S}\Delta^*_j(1-z_j)\geq h^*$ is a valid inequality for $\emph{conv}(\mathcal{Q})$ which is violated by $(y^*,z^*)$. Furthermore, when $a_j=m_S$ for all $j\in S$, separation over all inequalities of the type \begin{eqnarray} y+\sum_{j\in R}\delta_jz_j+\sum_{j\in S}\Delta_j(1-z_j)\geq h~~~(\delta,\Delta,h)\in \mathbb{R}^R_+\times \mathbb{R}^S_+\times \mathbb{R},\label{RStype}\end{eqnarray} can be accomplished in polynomial time. \end{thm}

\begin{proof} Suppose that $y^*+LP^*< 0$ and that $(\delta^*,\Delta^*,h^*)$ is an optimal solution to $(\ref{RSsepopt})$. Then, by Lemma \ref{RSprojlemma}, we know that $(\delta^*,\Delta^*,h^*)\in \mathcal{A}(R)\subset \mathcal{G}(R)$, and so by Corollary \ref{RSvalidthm}, $y+\sum_{j\in R}\delta^*_jz_j+\sum_{j\in S}\Delta^*_j(1-z_j)\geq h^*$ is a valid inequality for $\text{conv}(\mathcal{Q})$. Since $y^*+LP^*<0$, it follows that $(y^*,z^*)$ violates this inequality and so $(y^*,z^*)\notin \text{conv}(\mathcal{Q})$. When $a_j=m_S$ for all $j\in S$, we have that $\mathcal{G}(R)=\mathcal{A}(R)= \text{proj}_{\mathbb{R}^R\times \mathbb{R}^S\times \mathbb{R}}\Theta(R)$, and so by solving the linear program (\ref{RSsepopt}) (in polynomial time), one can separate over all inequalities of the type (\ref{RStype}).\end{proof}

Observe that the above algorithm yields a heuristic separation algorithm over all inequalities of the type $$y+\sum_{j\in [n]}\alpha_jz_j\geq \beta~~~~\forall~\alpha\in \mathbb{R}^n:\alpha_j\geq 0 \text{ if and only if } j\in R,$$ for a fixed subset $R$ of $[n]$. Furthermore, this separation algorithm is exact when $a_i=m_S>0$ for all $i\in S$.

\section{Conclusion}\label{concl}

In this paper, our main purpose is to recognize some structural properties of the convex hull of the mixing set subject to a knapsack constraint arising in chanced-constrained programming. We start off by characterizing the set of all the valid inequalities for this polyhedron. This characterization helps us in two ways. Firstly, it helps us find a new class of explicit facet-defining inequalities that subsumes the class of strengthened star-inequalities, which were the {\it only} known explicit class previously known for the general knapsack constraint. Secondly, it helps us in finding a polynomial time heuristic separation algorithm for the polyhedron. We also give necessary conditions for the facet-defining inequalities of the polyhedron.

A complete characterization of the facet-defining inequalities of the convex hull of the mixing set subject to a knapsack constraint arising in chanced-constrained programming remains an open problem. We also intend to perform computational experiments with the proposed inequalities to measure their effectiveness.

\end{document}